\documentclass[twoside,reqno,A4]{amsart}

\newcommand{\qqdn}{\hspace*{-2.5mm}}

\newcommand{\+}{&\qqdn}%

\newcommand{\mb}[1]{\mathbb{#1}}

\newcommand{\binq}[2]{\genfrac{[}{]}{0mm}{0}{#1}{#2}}

\newcommand{\nnm}{\nonumber}
\newcommand{\be}{\begin{equation}}
\newcommand{\ee}{\end{equation}}
\newcommand{\ba}{\begin{array}}
\newcommand{\ea}{\end{array}}
\newcommand{\bmn}{\begin{eqnarray}}
\newcommand{\emn}{\end{eqnarray}}
\newcommand{\bnm}{\begin{eqnarray*}}
\newcommand{\enm}{\end{eqnarray*}}
\newcommand{\bln}{\begin{subequations}}
\newcommand{\eln}{\end{subequations}}

\newtheorem{thm}{Theorem}

\newcommand{\bbtm}[4]{\bibitem{kn:#1}{#2,}~\emph{#3}~{#4.}}
\newcommand{\cito}[1]{\cite{kn:#1}}

\def\SUB#1{\vskip .2in\leftline{\large\bf #1}\vskip .1in}

\begin{document}

\title{\large Recursion formulas of $q$-Appell functions}
\author{\large Xiaoxia Wang $^a$ and Chuanan Wei $^b$ }
\dedicatory{\large $^A$ Department of Mathematics,
            Shanghai University\\
            Shanghai 200444, P.\:R.\:China\\[2mm]
           $^B$ Department of Medical Informatics\\
            Hainan Medical University, Haikou 571199, China}
\thanks{
$^B$ Corresponding author;\\
E-mail: xiaoxiawang@shu.edu.cn (X. Wang) and weichuanan78@126.com(C. Wei);  \\
This work is supported by National Natural Science Foundations of China (11661032).}

\maketitle\thispagestyle{empty}
\markboth{Xiaoxia Wang--Chuanan Wei}%
        {Recursion formulas of $q-$Appell functions}
\begin{center}\parbox{120mm}{

Recently, Opps, Saad and Srivastava gave the recursion formulas of Appell's function $F_2$.
The first author of this paper then established
the recursion formulas for Appell functions $F_1, F_2, F_3$
and $F_4$ by the contiguous relations of hypergeometric series.
In this paper, the authors will present the recursion formulas for $q$-Appell functions $\Phi^{(1)}, \Phi^{(2)}, \Phi^{(3)}$
and $\Phi^{(4)}$ as the $q$-analogies of $F_1, F_2, F_3$ and $F_4$'s relations. \
\\[4mm]
}
\end{center}
\vspace*{7mm}
Appell functions \cito{appell} which are famous in the field of double hypergeometric functions
\cite{kn:exton,kn:slater,kn:sri-kar} are read as follows:
\bnm
\+\+F_1[a;b_1,b_2;c;x,y]:=\sum_{m,n\geq0}
\frac{(a)_{m+n}(b_1)_m(b_2)_n}{(c)_{m+n}}\frac{x^m}{m!}\frac{y^n}{n!};\\
\+\+F_2[a;b_1,b_2;c_1,c_2;x,y]:=\sum_{m,n\geq0}
\frac{(a)_{m+n}(b_1)_m(b_2)_n}{(c_1)_{m}(c_2)_n}\frac{x^m}{m!}\frac{y^n}{n!};\\
\+\+F_3[a_1,a_2;b_1,b_2;c;x,y]:=\sum_{m,n\geq0}
\frac{(a_1)_{m}(a_2)_n(b_1)_m(b_2)_n}{(c)_{m+n}}\frac{x^m}{m!}\frac{y^n}{n!};\\
\+\+F_4[a;b;c_1,c_2;x,y]:=\sum_{m,n\geq0}
\frac{(a)_{m+n}(b)_{m+n}}{(c_1)_{m}(c_2)_n}\frac{x^m}{m!}\frac{y^n}{n!}.
\enm
Then, Jackson \cite{kn:jackson-1,kn:jackson-2} first discussed the $q$-Appell functions $\Phi^{(1)}, \Phi^{(2)}, \Phi^{(3)}$
and $\Phi^{(4)}$:
\bnm
\+\+\Phi^{(1)}[a;b,b';c;x,y]
=\sum_{m,n\geq0}
\frac{(a;q)_{m+n}(b;q)_m(b';q)_n}
     {(q;q)_m(q;q)_n(c;q)_{m+n}}
     x^m y^n;\\[2mm]
\+\+\Phi^{(2)}[a;b,b';c,c';x,y]
=\sum_{m,n\geq0}
\frac{(a;q)_{m+n}(b;q)_m(b';q)_n}
     {(q;q)_m(q;q)_n(c;q)_m(c';q)_n}
     x^m y^n;
\enm
\bnm
\+\+\Phi^{(3)}[a,a';b,b';c;x,y]
=\sum_{m,n\geq0}
\frac{(a;q)_m(a';q)_n(b;q)_m(b';q)_n}
     {(q;q)_m(q;q)_n(c;q)_{m+n}}
     x^m y^n;\\[2mm]
\+\+\Phi^{(4)}[a;b;c,c';x,y]
=\sum_{m,n\geq0}
\frac{(a;q)_{m+n}(b;q)_{m+n}}
     {(q;q)_m(q;q)_n(c;q)_m(c';q)_n}
     x^m y^n;
\enm
which are the $q$-analogies of $F_1, F_2, F_3$ and $F_4$.

When $|q|<1$, the shifted factorial of infinite order is well-defined as
\[(x;q)_\infty
\::=\:\prod_{k=0}^{\infty}(1-xq^k) \quad\text{and}\quad
(x;q)_n\:=\:\frac{(x;q)_\infty}{(xq^n;q)_\infty}
\quad\text{for}\quad n\in\mb{Z}.\]



The research of recursion formulas of hypergeometric function are important and interesting.
Opps, Saad and Srivastava \cito{srivastava} established some recursion formulas for the function $F_2$ by the contiguous relations of the Gauss hypergeometric series
$_2F_1$, and then applied the relations to radiation field problem. Wang \cito{wang} gave the recursion
formulas for Appell's four functions $F_1, F_2, F_3$ and $F_4$ which including Opps, Saad and Srivastava's results.
Chu and Wang \cite{kn:chu3,kn:chu-wang2}
have reviewed many hypergeometric summation formulas by the recursion formulas which are
obtained by Abel's lemma on summation by parts.
In this paper, the authors will present the recursion formulas for $q$-Appell functions $\Phi^{(1)}, \Phi^{(2)}, \Phi^{(3)}$
and $\Phi^{(4)}$, and all the results are verified by $\textbf{Mathematica}$.

\SUB{1. Recursion formulas of $\Phi^{(1)}$}
In this part, we will present the recursion formulas for $q$-Appell function $\Phi^{(1)}$
with five theorems as follows.
First, we present the recursion formulas of $\Phi^{(1)}$ with the numerator parameter $a$.
\begin{thm}[The recursion formulas of $\Phi^{(1)}$ with parameter $a$]
\bmn\label{f1-a-a}
\Phi^{(1)}[aq^n;b,b';c;x,y]=\Phi^{(1)}[a;b,b';c;x,y]\++\+\frac{ax(1-b)}{(1-c)}\sum_{k=1}^nq^{k-1}\Phi^{(1)}[aq^k;b q,b';c q;x,y]\nnm\\
\++\+ \frac{ay(1-b')}{(1-c)}\sum_{k=1}^n q^{k-1}\Phi^{(1)}[aq^k;b,b' q;c q;x q,y];\label{f1-a-1}\\
\Phi^{(1)}[aq^{-n};b,b';c;x,y]=\Phi^{(1)}[a;b,b';c;x,y]\+-\+\frac{ax(1-b)}{(1-c)}\sum_{k=1}^n q^{-k}\Phi^{(1)}[aq^{1-k};b q,b';c q;x,y]\nnm\\
\+-\+ \frac{ay(1-b')}{(1-c)}\sum_{k=1}^n q^{-k}\Phi^{(1)}[aq^{1-k}; b, b' q; c q;x q, y].\label{f1-a-2}
\emn
\end{thm}

\begin{proof}
From the definition of $q-$Appell's function $\Phi^{(1)}$ and the transformation
$(aq; q)_{m+n}=(a; q)_{m+n}[1+\frac{a(1-q^m)}{1-a}+\frac{aq^m(1-q^n)}{1-a}]$, we get the following contiguous relation:
\bmn
\Phi^{(1)}[aq,b,b';c;x,y]=\Phi^{(1)}[a,b,b';c;x,y]\++\+\frac{ax(1-b)}{(1-c)}\Phi^{(1)}[aq,bq,b';cq;x,y]\nnm\\
                        \++\+\frac{ay(1-b')}{(1-c)}\Phi^{(1)}[aq,b,b'q;cq;xq,y].\label{r-2}
\emn
Applying the above contiguous relation on function $\Phi^{(1)}$ with parameter $aq^2$, we have
\bnm
\Phi^{(1)}[aq^2,b,b';c;x,y]\+=\+\Phi^{(1)}[aq,b,b';c;x,y]+\frac{aqx(1-b)}{(1-c)}\Phi^{(1)}[aq^2,bq,b';cq;x,y]\nnm\\
                        \++\+\frac{aqy(1-b')}{(1-c)}\Phi^{(1)}[aq^2,b,b'q;cq;xq,y]=\Phi^{(1)}[a,b,b';c;x,y]\\
\++\+\frac{ax(1-b)}{(1-c)}\Big\{\Phi^{(1)}[aq,bq,b';cq;x,y]+q\Phi^{(1)}[aq^2,bq,b';cq;x,y]\Big\}\\
 \++\+\frac{ay(1-b')}{(1-c)}\Big\{\Phi^{(1)}[aq,b,b'q;cq;xq,y]+q\Phi^{(1)}[aq^2,b,b'q;cq;xq,y]\Big\}.
\enm
Iterating this computation on $\Phi^{(1)}$ for $n-$times,
we get the recursion formula \eqref{f1-a-1} with parameter $aq^n$.
Performing the replacement $a\to aq^{-1}$ in the contiguous relation \eqref{r-2}, we have
\bmn
\Phi^{(1)}[aq^{-1},b,b';c;x,y]=\Phi^{(1)}[a,b,b';c;x,y]\+-\+\frac{ax(1-b)}{q(1-c)}\Phi^{(1)}[a,bq,b';cq;x,y]\nnm\\
                        \+-\+\frac{ay(1-b')}{q(1-c)}\Phi^{(1)}[a,b,b'q;cq;xq,y].\label{r-2-1}
\emn
Applying this contiguous relation on function $\Phi^{(1)}$ for $n$-times,
we obtain the recursion formula \eqref{f1-a-2} as same as we have done in the proof of \eqref{f1-a-1}.
\end{proof}

By the known contiguous relations \eqref{r-2} and \eqref{r-2-1}, we can express the hypergeometric functions $\Phi^{(1)}$
with $aq^n$ and $aq^{-n}$ in another expressions.
\begin{thm}[The recursion formulas of $\Phi^{(1)}$ with parameter $a$ in another expression]
\bmn\label{a-1}
\Phi^{(1)}[aq^n,b,b';c;x,y]\+=\+\sum_{k=0}^n\sum_{i=0}^{k}\binq{n}{k}\binq{k}{i}
\frac{(b;q)_{k-i}(b';q)_i}{(c;q)_{k}}q^{2{k\choose 2}}a^k x^{k-i}y^i\nnm\\[2mm]
\+\times\+\Phi^{(1)}[aq^k,bq^{k-i},b'q^i;cq^k;xq^i,y];\label{f1-a-1-1}\\[2mm]
\Phi^{(1)}[aq^{-n},b,b';c;x,y]\+=\+\sum_{k=0}^n\sum_{i=0}^{k}\binq{n}{k}\binq{k}{i}
\frac{(b;q)_{k-i}(b';q)_{i}}{(c;q)_{k}}q^{{k\choose 2}-nk}(-a)^k x^{k-i}y^{i}\nnm\\[2mm]
\+\times\+\Phi^{(1)}[a,bq^{k-i},b'q^{i};cq^k;xq^{i},y]. \label{f1-a-2-2}
\emn
\end{thm}
\begin{proof}
Here, we just prove the recursion formula \eqref{f1-a-1-1} by the induction method for example.
The relation \eqref{f1-a-2-2} can be proved by the similarly method.
When $n=1$, formula \eqref{f1-a-1-1} reduces to relation \eqref{r-2} obviously. Suppose that the result \eqref{f1-a-1-1}
is true for $n\leq t$ and the recursion formula \eqref{f1-a-1-1} reads as follows when $n=t$:
\bnm
\Phi^{(1)}[aq^t,b,b';c;x,y]=\sum_{k=0}^t\sum_{i=0}^{k}\binq{t}{k}\binq{k}{i}
\frac{(b;q)_{k-i}(b';q)_i}{(c;q)_{k}}q^{2{k\choose 2}}a^k x^{k-i}y^i\:\Phi^{(1)}[aq^k,bq^{k-i},b'q^i;cq^k;xq^i,y].
\enm
Now we only need to confirm the correction of \eqref{f1-a-1-1} with $n=t+1$. Performing the replacement $a \to aq$ in the above relation, we have
\bnm
\+\+\Phi^{(1)}[aq^{t+1},b,b';c;x,y]\\
\+=\+\sum_{k=0}^t\sum_{i=0}^{k}\binq{t}{k}\binq{k}{i}
\frac{(b;q)_{k-i}(b';q)_i}{(c;q)_{k}}q^{2{k\choose 2}}{(aq)}^k x^{k-i}y^i\Phi^{(1)}[aq^{k+1},bq^{k-i},b'q^i;cq^k;xq^i,y]\\
\+=\+\sum_{k=0}^t\sum_{i=0}^{k}\binq{t}{k}\binq{k}{i}
\frac{(b;q)_{k-i}(b';q)_i}{(c;q)_{k}}q^{2{k\choose 2}}{(aq)}^k x^{k-i}y^i
\Big\{\Phi^{(1)}[aq^k,bq^{k-i},b'q^i;cq^k;xq^i,y]\\
\++\+\frac{aq^kxq^i(1-bq^{k-i})}{(1-cq^k)}\Phi^{(1)}[aq^{k+1},bq^{1+k-i},b'q^i;cq^{k+1};xq^i,y]\nnm\\
                        \++\+\frac{aq^ky(1-b'q^i)}{(1-cq^k)}\Phi^{(1)}[aq^{k+1},bq^{k-i},b'q^{i+1};cq^{k+1};xq^{i+1},y]\Big\}.
\enm
In the second equality, we have applied the contiguous relation
\eqref{r-2} with the replacements $a\to aq^k, b\to bq^{k-i},
b'\to b'q^i$, $c\to cq^k$ and $x\to xq^i$.
Simplifying the above result, we have
\bmn
\+\+\Phi^{(1)}[aq^{t+1},b,b';c;x,y]\label{proof-1}\\
\+=\+\sum_{k=0}^t\sum_{i=0}^{k}\binq{t}{k}\binq{k}{i}
\frac{(b;q)_{k-i}(b';q)_i}{(c;q)_{k}}q^{2{k\choose 2}+k}{a}^k x^{k-i}y^i\Phi^{(1)}[aq^{k},bq^{k-i},b'q^i;cq^k;xq^i,y]\nnm\\
\++\+\sum_{k=0}^t\sum_{i=0}^{k}\binq{t}{k}\binq{k}{i}
\frac{(b;q)_{k+1-i}(b';q)_i}{(c;q)_{k+1}}q^{2{k\choose 2}+2k+i}{a}^{k+1} x^{k+1-i}y^i\Phi^{(1)}[aq^{k+1},bq^{k+1-i},b'q^i;cq^{k+1};xq^i,y]\nnm\\
\++\+\sum_{k=0}^t\sum_{i=0}^{k}\binq{t}{k}\binq{k}{i}
\frac{(b;q)_{k-i}(b';q)_{i+1}}{(c;q)_{k+1}}q^{2{k\choose 2}+2k}{a}^{k+1} x^{k-i}y^{i+1}\Phi^{(1)}[aq^{k+1},bq^{k-i},b'q^{i+1};cq^{k+1};xq^{i+1},y].\nnm
\emn
Extracting the coefficient of
\[\frac{(b;q)_{k-i}(b';q)_i}{(c;q)_{k}}q^{2{k\choose 2}}{a}^k x^{k-i}y^i\:\Phi^{(1)}[aq^{k},bq^{k-i},b'q^i;cq^k;xq^i,y]\]
on the right-hand side of \eqref{proof-1} and
applying the relations
$q^k \binq{n-1}{k}+ \binq{n-1}{k-1}=\binq{n}{k}$
and $\binq{n}{m}\equiv0$ with $m>n$ or $m<0$,
we have
\[
q^k\binq{t}{k}\binq{k}{i}+q^i\binq{t}{k-1}\binq{k-1}{i}+\binq{t}{k-1}\binq{k-1}{i-1}=\binq{t+1}{k}\binq{k}{i},
\]
which is exactly the coefficient of $\Phi^{(1)}$ with $n=t+1$
in \eqref{f1-a-1-1}.
Now, we certified the result \eqref{f1-a-1-1}.
Applying the relation \eqref{r-2-1}, we can confirm the recursion formula \eqref{f1-a-2-2}
by induction method too. This completes the proof of this theorem.
\end{proof}

Second, we establish the recursion formulas of $\Phi^{(1)}$ about the numerator parameter $b$. The recursion formulas about
$b'$ can be obtained by the similar method.
\begin{thm}[The recursion formulas of $\Phi^{(1)}$ with parameter $b$]
\bmn
\+\+\Phi^{(1)}[a,bq^n,b';c;x,y]=\Phi^{(1)}[a,b,b';c;x,y]+\frac{bx(1-a)}{(1-c)}\sum_{k=1}^nq^{k-1}\Phi^{(1)}[aq,bq^k,b';cq;x,y];\label{f1-b-1}\\
\+\+\Phi^{(1)}[a,bq^{-n},b';c;x,y]=\Phi^{(1)}[a,b,b';c;x,y]-\frac{bx(1-a)}{(1-c)}\sum_{k=1}^nq^{-k}\Phi^{(1)}[aq,bq^{1-k},b';cq;x,y].\label{f1-b-2}
\emn
\end{thm}

\begin{proof}
From the definition of $q$-Appell function $\Phi^{(1)}$ and $(bq;q)_m=(b;q)_m[1+\frac{b}{1-b}(1-q^m)]$,
we can easily get the following contiguous relation:
\bmn
\Phi^{(1)}[a,bq,b';c;x,y]=\Phi^{(1)}[a,b,b';c;x,y]+\frac{bx(1-a)}{(1-c)}\Phi^{(1)}[aq,bq,b';cq;x,y].\label{r-1}
\emn
Replacing $b\to bq$ in the above relation, we have
\bnm
\+\Phi^{(1)}\+[a,bq^2,b';c;x,y]=\Phi^{(1)}[a,bq,b';c;x,y]+\frac{bqx(1-a)}{(1-c)}\Phi^{(1)}[aq,bq^2,b';cq;x,y]\\
\+=\+\Phi^{(1)}[a,b,b';c;x,y]+\frac{b x(1-a)}{(1-c)}\Big\{\Phi^{(1)}[aq,bq,b';cq;x,y]
+q\Phi^{(1)}[aq,bq^2,b';cq;x,y]\Big\},
\enm
where we have applied the contiguous relation \eqref{r-1} in the second equality.
Iterating this method on $\Phi^{(1)}$ for $n$ times, we have
\bnm
\+\Phi^{(1)}\+[a,bq^n,b';c;x,y]=\Phi^{(1)}[a,bq^{n-1},b';c;x,y]+\frac{bq^{n-1}x(1-a)}{(1-c)}\Phi^{(1)}[aq,bq^n,b';cq;x,y]\\
\+=\+\Phi^{(1)}[a,b,b';c;x,y]+\frac{b x(1-a)}{(1-c)}\Big\{\Phi^{(1)}[aq,bq,b';cq;x,y]+\cdots
+q^{n-1}\Phi^{(1)}[aq,bq^n,b';cq;x,y]\Big\}\\
\+=\+\Phi^{(1)}[a,b,b';c;x,y]+\frac{bx(1-a)}{(1-c)}\sum_{k=1}^nq^{k-1}\Phi^{(1)}[aq,bq^k,b';cq;x,y],
\enm
which is exactly the recursion formula \eqref{f1-b-1}.

Performing the replacement $b\to b/q$ in relation \eqref{r-1}, we have
\bmn
\Phi^{(1)}[a,b/q,b';c;x,y]=\Phi^{(1)}[a,b,b';c;x,y]-\frac{b x(1-a)}{q(1-c)}\Phi^{(1)}[aq,b,b';cq;x,y]\label{r-1-1}.
\emn
Applying this relation on function $\Phi^{(1)}$ with the parameter $bq^{-n}$ for $n-$times,
we arrive at the recursion formula \eqref{f1-b-2}. This completes the proof of this theorem.
\end{proof}

In fact, we have another expression of recursion formulas for hypergeometric
functions $\Phi^{(1)}$ with the parameters $bq^n$ and $bq^{-n}$.
\begin{thm} [The recursion formulas of $\Phi^{(1)}$ with parameter $b$ in another expression]
\bmn
\Phi^{(1)}[a,bq^n,b';c;x,y]\+=\+\sum_{k=0}^n\binq{n}{k}q^{2{k\choose 2}}\frac{(bx)^k(a;q)_k}{(c;q)_k}\Phi^{(1)}[aq^k,bq^k,b';cq^k;x,y];\label{f1-b-1-1}\\
\Phi^{(1)}[a,bq^{-n},b';c;x,y]\+=\+\sum_{k=0}^n\binq{n}{k}q^{{k\choose 2}-nk}\frac{(-bx)^k(a;q)_k}{(c;q)_k}\Phi^{(1)}[aq^k,b,b';cq^k;x,y].\label{f1-b-2-2}
\emn
\end{thm}
\begin{proof}
This theorem can be proved by inductive method as we have done in Theorem \ref{a-1}. Here, we just prove the recursion formula
\eqref{f1-b-1-1} for example. The formula \eqref{f1-b-2-2} can be proved by the similarly method.
When $n=1$, the formula \eqref{f1-b-1-1} is exactly \eqref{r-1}. Suppose that the recursion formula
\eqref{f1-b-1-1} is true for $n\leq t$, and \eqref{f1-b-1-1} reduces to the following result when $n=t$:
\bnm
\Phi^{(1)}[a,bq^t,b';c;x,y]\+=\+\sum_{k=0}^t\binq{t}{k}q^{2{k\choose 2}}\frac{(bx)^k(a;q)_k}{(c;q)_k}\Phi^{(1)}[aq^k,bq^k,b';cq^k;x,y].
\enm
Performing the replacement $b\to bq$ in the above result, we have
\bnm
\Phi^{(1)}[a,bq^{t+1},b';c;x,y]\+=\+\sum_{k=0}^{t}\binq{t}{k}q^{2{k\choose 2}}\frac{(bqx)^k(a;q)_k}{(c;q)_k}\Phi^{(1)}[aq^k,bq^{k+1},b';cq^k;x,y]\\
\+=\+\sum_{k=0}^{t}\binq{t}{k}q^{2{k\choose 2}}\frac{(bqx)^k(a;q)_k}{(c;q)_k}
\Big\{
\Phi^{(1)}[aq^k,bq^{k},b';cq^k;x,y]\\
\++\+\frac{bq^k x(1-aq^k)}{1-cq^k}\Phi^{(1)}[a^{k+1},b^{k+1},b';c^{k+1};x,y]\Big\}\\
\+=\+\sum_{k=0}^{t}\binq{t}{k}q^{2{k\choose 2}+k}\frac{(bx)^k(a;q)_k}{(c;q)_k}\Phi^{(1)}[aq^k,bq^{k},b';cq^k;x,y]\\
\++\+\sum_{k=0}^{t}\binq{t}{k}q^{2{k\choose 2}+2k}\frac{(bx)^{k+1}(a;q)_{k+1}}{(c;q)_{k+1}}\Phi^{(1)}[aq^{k+1},bq^{k+1},b';cq^{k+1};x,y]\\
\+=\+\sum_{k=0}^{t+1}\Big\{q^k\binq{t}{k}+\binq{t}{k-1}\Big\}q^{2{k\choose 2}}\frac{(bx)^k(a;q)_k}{(c;q)_k}
\Phi^{(1)}[aq^k,bq^{k},b';cq^k;x,y]\\
\+=\+\sum_{k=0}^{t+1}\binq{t+1}{k}q^{2{k\choose 2}}\frac{(bx)^k(a;q)_k}{(c;q)_k}
\Phi^{(1)}[aq^k,bq^{k},b';cq^k;x,y],
\enm
where, we have applied the contiguous relation \eqref{r-1} in the second equality and
$q^k\binq{t}{k}+\binq{t}{k-1}=\binq{t+1}{k}$ and $\binq{m}{n}\equiv 0$ when $n>m$ and $n<0$ in the fifth equality.
Now we completes the proof of the recursion formula \eqref{f1-b-1-1}. Applying the contiguous relation \eqref{r-1-1},
we can get the recursion formula \eqref{f1-b-2-2} by the similarly method.
Here, we completes the proof of this theorem.
\end{proof}

Finally, we present the recursion formulas of $\Phi^{(1)}$ about the denominator parameter $c$.
\begin{thm} [The recursion formulas of $\Phi^{(1)}$ with parameter $c$]
\bmn
\+\+\Phi^{(1)}[a,b,b';cq^{-n};x,y]
=\frac{1}{(q/c;q)_n}\sum_{k=0}^n\binq{n}{k}(-c)^{k-n}q^{{{n+1-k}\choose 2}-1}\Phi^{(1)}[a; b,b';c;xq^k,yq^k];\label{f1-c}\\[2mm]
\+\+\Phi^{(1)}[a,b,b';cq^{n};x,y]
=\sum_{k=0}^n\binq{n}{k}c^{k}q^{2{{k}\choose 2}}(cq^k;q)_{n-k}\Phi^{(1)}[a; b,b';cq^k;xq^k,yq^k].\label{f1-c-2}
\emn
\end{thm}
\begin{proof}
From the definition of $q$-Appell function $\Phi^{(1)}$ and the transformation
\[\frac{1}{(c/q;q)_{m+n}}=\frac{1}{(c;q)_{m+n}}\Big\{\frac{c}{c-q}q^{m+n}-\frac{q}{c-q}\Big\},\]
we get the following contiguous relation:
\bmn
\Phi^{(1)}[a,b,b';c/q;x,y]=\frac{1}{1-q/c}\Phi^{(1)}[a,b,b';c;xq,yq]
                        -\frac{q/c}{1-q/c}\Phi^{(1)}[a,b,b';c;x,y].\label{c-1}
\emn

Obviously, the relation \eqref{f1-c} is exactly right when $n=1$. Suppose that the result \eqref{f1-c} is correct when $n=t$ as follows:
\bnm
\Phi^{(1)}[a,b,b';cq^{-t};x,y]
=\frac{1}{(q/c;q)_t}\sum_{k=0}^t(-c)^{k-t}q^{{{t+1-k}\choose 2}-1}\binq{t}{k}\Phi^{(1)}[a,b,b';c;xq^k,yq^k].
\enm
we just need to proof that the result is right with $n=t+1$ by the induction method.
Performing the replacement $c\to c/q$ in the above identity, we have
\bnm
\+\+\Phi^{(1)}[a,b,b';cq^{-t-1};x,y]
=\frac{1}{(q^2/c;q)_t}\sum_{k=0}^t(-c/q)^{k-t}q^{{{t+1-k}\choose 2}-1}\binq{t}{k}\Phi^{(1)}[a,b,b';c/q;xq^k,yq^k]\\
\+\+=\frac{1}{(q^2/c;q)_t}\sum_{k=0}^t(-c/q)^{k-t}q^{{{t+1-k}\choose 2}-1}\binq{t}{k}\\
\+\+\quad\times\Big\{
\frac{1}{1-q/c}\Phi^{(1)}[a,b,b';c;xq^{k+1},yq^{k+1}]
                        -\frac{q/c}{1-q/c}\Phi^{(1)}[a,b,b';c;xq^k,yq^k]\Big\}\\
\+\+=\frac{1}{(q/c;q)_{t+1}}\sum_{k=0}^t(-c)^{k-t}q^{{{t+1-k}\choose 2}+t-k-1}\binq{t}{k}\Phi^{(1)}[a,b,b';c;xq^{k+1},yq^{k+1}]\\
\+\++\frac{1}{(q/c;q)_{t+1}}\sum_{k=0}^t(-c)^{k-t-1}q^{{{t+1-k}\choose 2}+t-k}\binq{t}{k}\Phi^{(1)}[a,b,b';c;xq^k,yq^k]\\
\+\+=\frac{1}{(q/c;q)_{t+1}}\sum_{k=0}^t(-c)^{k-t-1}q^{{{t+2-k}\choose 2}+t-k}\Big\{\binq{t}{k-1}+q^{k-t-1}\binq{t}{k}\Big\}\Phi^{(1)}[a,b,b';c;xq^k,yq^k]\\
\+\+=\frac{1}{(q/c;q)_{t+1}}\sum_{k=0}^{t+1}(-c)^{k-t-1}q^{{{t+2-k}\choose 2}-1}\binq{t+1}{k}\Phi^{(1)}[a,b,b';c;xq^k,yq^k],
\enm
where, we have applied the transformation
\[
q^{k-t-1}\binq{t}{k}+\binq{t}{k-1}=q^{k-t-1}\binq{t+1}{k}.
\]
Performing $c\to cq$ in contiguous relation \eqref{c-1}, we get
\bnm
\Phi^{(1)}[a,b,b';cq;x,y]=(1-c)\Phi^{(1)}[a,b,b';c;x,y]
                        +c\:\Phi^{(1)}[a,b,b';cq;xq,yq].
\enm
Applying this contiguous relation, we can arrive at the recursion formula \eqref{f1-c-2} by induction method.
This completes the proof of this theorem.
\end{proof}

\SUB{2. Recursion formulas of $\Phi^{(2)}$ }
In this part, we will list the recursion formulas of $q$-Appell function $\Phi^{(2)}$ with the parameters $a$, $b$ and $c$.
All the theorems can be proved by the similarly method as we have done in part one.

By the definition of $\Phi^{(2)}$, we can get the following two contiguous relations of $\Phi^{(2)}$ with parameter $a$ as:

\bnm
\Phi^{(2)}[aq,b,b';c, c';x,y]=\Phi^{(2)}[a,b,b';c,c';x,y]\++\+\frac{ax(1-b)}{1-c}\Phi^{(2)}[aq,bq,b';cq, c';x,y]\\
                        \++\+\frac{ay(1-b')}{1-c'}\Phi^{(2)}[aq,b,b'q; c,c'q;xq,y];\\
\Phi^{(2)}[aq^{-1},b,b';c,c';x,y]=\Phi^{(2)}[a,b,b';c,c';x,y]\+-\+\frac{ax(1-b)}{q(1-c)}\Phi^{(2)}[a,bq,b';cq,c';x,y]\\
                        \+-\+\frac{ay(1-b')}{q(1-c')}\Phi^{(2)}[a,b,b'q;c,c'q;xq,y].
\enm

From the above relations, we can establish the recursion formulas of $\Phi^{(2)}$ with parameter $a$ in the following two theorems.
\begin{thm} [The recursion formulas of $\Phi^{(2)}$ with parameter $a$]
\bnm
\Phi^{(2)}[aq^n,b,b';c,c';x,y]=\Phi^{(2)}[a,b,b';c,c';x,y]\++\+\frac{ax(1-b)}{(1-c)}\sum_{k=1}^nq^{k-1}\Phi^{(2)}[aq^k,b q,b';c q,c';x,y]\\
\++\+ \frac{ay(1-b')}{(1-c')}\sum_{k=1}^n q^{k-1}\Phi^{(2)}[aq^k,b,b' q;c, c'q;x q,y];\\
\Phi^{(2)}[aq^{-n},b,b';c,c';x,y]=\Phi^{(2)}[a,b,b';c,c';x,y]\+-\+\frac{ax(1-b)}{(1-c)}\sum_{k=1}^n q^{-k}\Phi^{(2)}[aq^{1-k},b q,b';c q,c';x,y]\\
\+-\+ \frac{ay(1-b')}{(1-c')}\sum_{k=1}^n q^{-k}\Phi^{(2)}[aq^{1-k}, b, b' q; c, c'q;x q, y].
\enm
\end{thm}

\begin{thm} [The recursion formulas of $\Phi^{(2)}$ with parameter $a$ in another expression]
\bnm
\Phi^{(2)}[aq^n,b,b';c,c';x,y]\+=\+\sum_{k=0}^n\sum_{i=0}^{k}\binq{n}{k}\binq{k}{i}
\frac{(b;q)_{k-i}(b';q)_i}{(c;q)_{k-i}(c';q)_{i}}q^{2{k\choose 2}}a^k x^{k-i}y^i\\
\+\times\+\Phi^{(2)}[aq^k,bq^{k-i},b'q^i;cq^{k-i},c'q^i;xq^i,y];\\[2mm]
\Phi^{(2)}[aq^{-n},b,b';c,c';x,y]\+=\+\sum_{k=0}^n\sum_{i=0}^{k}\binq{n}{k}\binq{k}{i}
\frac{(b;q)_{k-i}(b';q)_{i}}{(c;q)_{k-i}(c';q)_{i}}q^{{k\choose 2}-nk}(-a)^k x^{k-i}y^{i}\\
\+\times\+\Phi^{(2)}[a,bq^{k-i},b'q^{i};cq^{k-i},c'q^{i};xq^{i},y];
\enm
\end{thm}

Applying the contiguous relations of $\Phi^{(2)}$ with parameter $b$ as follows
\bnm
\+\+\Phi^{(2)}[a,bq,b';c, c';x,y]=\Phi^{(2)}[a,b,b';c,c';x,y]+\frac{bx(1-a)}{1-c}\Phi^{(2)}[aq,bq,b';cq, c';x,y];\\
\+\+\Phi^{(2)}[a,bq^{-1},b';c,c';x,y]=\Phi^{(2)}[a,b,b';c,c';x,y]-\frac{bx(1-a)}{q(1-c)}\Phi^{(2)}[aq,b,b';cq,c';x,y],
\enm
we can establish the recursion formulas with parameter $b$ in two expressions in the following two theorems.
\begin{thm} [The recursion formulas of $\Phi^{(2)}$ with the parameter $b$]
\bnm
\+\+\Phi^{(2)}[a,bq^n,b';c,c';x,y]=\Phi^{(2)}[a,b,b';c,c';x,y]+\frac{bx(1-a)}{(1-c)}\sum_{k=1}^nq^{k-1}\Phi^{(2)}[aq,bq^k,b';cq, c';x,y];\\
\+\+\Phi^{(2)}[a,bq^{-n},b';c,c';x,y]=\Phi^{(2)}[a,b,b';c,c';x,y]-\frac{bx(1-a)}{(1-c)}\sum_{k=1}^nq^{-k}\Phi^{(2)}[aq,bq^{1-k},b';cq,c';x,y].
\enm
\end{thm}

\begin{thm} [The recursion formulas of $\Phi^{(2)}$ with the parameter $b$ in another expression]
\bnm
\Phi^{(2)}[a,bq^n,b';c,c';x,y]\+=\+\sum_{k=0}^n\binq{n}{k}q^{2{k\choose 2}}\frac{(bx)^k(a;q)_k}{(c;q)_k}\Phi^{(2)}[aq^k,bq^k,b';cq^k,c';x,y];\\
\Phi^{(2)}[a,bq^{-n},b';c,c';x,y]\+=\+\sum_{k=0}^n\binq{n}{k}q^{{k\choose 2}-nk}\frac{(-bx)^k(a;q)_k}{(c;q)_k}\Phi^{(2)}[aq^k,b,b';cq^k,c';x,y].
\enm
\end{thm}

\begin{thm} [the recursion formulas with the parameter $c$]
\bnm
\+\+\Phi^{(2)}[a,b,b';cq^{-n},c';x,y]
=\frac{1}{(q/c;q)_n}\sum_{k=0}^n\binq{n}{k}(-c)^{k-n}q^{{{n+1-k}\choose 2}-1}\Phi^{(2)}[a,b,b';c,c';xq^k,y],\\
\+\+\Phi^{(2)}[a,b,b';cq^{n},c';x,y]
=\sum_{k=0}^n\binq{n}{k}c^{k}q^{2{{k}\choose 2}}(cq^k;q)_{n-k}\Phi^{(2)}[a; b,b';cq^k,c';xq^k,y].
\enm
\end{thm}
The above theorem can be proved by the following contiguous relations
\bnm
\+\+\Phi^{(2)}[a,b,b';cq, c';x,y]=(1-c)\Phi^{(2)}[a,b,b';c,c';x,y]+c\Phi^{(2)}[a,b,b';cq, c';xq,yq];\\
\+\+\Phi^{(2)}[a,b,b';cq^{-1},c';x,y]=\frac{1}{1-q/c}\Phi^{(2)}[a,b,b';c,c';xq,yq]-\frac{q/c}{1-q/c}\Phi^{(2)}[a,b,b';c,c';x,y],
\enm
which can be established easily by the definition of $\Phi^{(2)}$.
\SUB{3. Recursion formulas of $\Phi^{(3)}$ }
In this part, we will present the recursion formulas for $q-$Appell's hypergeometric function $\Phi^{(3)}$ with the parameters $b$ and $c$.
Applying the following two contiguous relations of $\Phi^{(3)}$ with parameter $b$,
\bnm
\+\+\Phi^{(3)}[a,a',bq,b';c;x,y]=\Phi^{(3)}[a,a',b,b';c;x,y]+\frac{bx(1-a)}{1-c}\Phi^{(3)}[aq,a',bq,b';cq, c';x,y];\\
\+\+\Phi^{(3)}[a,a',bq^{-1},b';c,c';x,y]=\Phi^{(3)}[a,a',b,b';c,c';x,y]-\frac{bx(1-a)}{q(1-c)}\Phi^{(3)}[aq,b,b';cq,c';x,y],
\enm
we can establish the recursion formulas with parameter $b$ in two expressions in the following two theorems.
\begin{thm} [The recursion formulas with parameter $b$]
\bnm
\+\+\Phi^{(3)}[a,a',bq^n,b';c;x,y]=\Phi^{(3)}[a,a',b,b';c;x,y]+\frac{bx(1-a)}{(1-c)}\sum_{k=1}^nq^{k-1}\Phi^{(3)}[aq,a';bq^k,b';cq;x,y];\\
\+\+\Phi^{(3)}[a,a';bq^{-n},b';c;x,y]=\Phi^{(3)}[a,a';b,b';c;x,y]-\frac{bx(1-a)}{(1-c)}\sum_{k=1}^nq^{-k}\Phi^{(3)}[aq,a';bq^{1-k},b';cq;x,y].
\enm
\end{thm}

\begin{thm} [The recursion formulas with parameter $b$ in another expression]
\bnm
\Phi^{(3)}[a,a';bq^n,b';c;x,y]\+=\+\sum_{k=0}^n\binq{n}{k}q^{2{k\choose 2}}\frac{(bx)^k(a;q)_k}{(c;q)_k}\Phi^{(3)}[aq^k,a';bq^k,b';cq^k,c';x,y];\\
\Phi^{(3)}[a,a';bq^{-n},b';c;x,y]\+=\+\sum_{k=0}^n\binq{n}{k}q^{{k\choose 2}-nk}\frac{(-bx)^k(a;q)_k}{(c;q)_k}\Phi^{(3)}[aq^k,a';b,b';cq^k;x,y].
\enm
\end{thm}

\begin{thm} [The recursion formulas with parameter $c$]
\bnm
\+\+\Phi^{(3)}[a,a';b,b';cq^{-n};x,y]
=\frac{1}{(q/c;q)_n}\sum_{k=0}^n\binq{n}{k}(-c)^{k-n}q^{{{n+1-k}\choose 2}-1}\:\Phi^{(3)}[a,a';b,b';c;xq^k,yq^k]\\
\+\+\Phi^{(3)}[a,a';b,b';cq^{n};x,y]
=\sum_{k=0}^n\binq{n}{k}c^{k}q^{2{{k}\choose 2}}(cq^k;q)_{n-k}\:\Phi^{(3)}[a,a'; b,b';cq^k;xq^k,yq^k].
\enm
\end{thm}
This theorem can be proved by the contiguous relation
\bnm
\+\+\Phi^{(3)}[a,a',b,b';cq; x,y]=(1-c)\Phi^{(3)}[a,a',b,b';c;x,y]+c\Phi^{(3)}[a,a',b,b';cq;xq,yq];\\
\+\+\Phi^{(3)}[a,a',b,b';cq^{-1};x,y]=\frac{1}{1-q/c}\Phi^{(3)}[a,a',b,b';c;xq,yq]-\frac{q/c}{1-q/c}\Phi^{(3)}[a,a',b,b';c;x,y].
\enm
All the recursion formulas of $\Phi^{(3)}$ are established with the similarly method as the results of function $\Phi^{(1)}$. Here, we will present with no details.
\SUB{4. Recursion formulas of $\Phi^{(4)}$ }
Here, we present the recursion formulas of $q$-Appell function $\Phi^{(4)}$ about parameters $a$ and $c$ by different expressions.

Applying the following two contiguous relations of $\Phi^{(4)}$ with parameter $a$,
\bnm
\Phi^{(4)}[aq,b;c, c';x,y]=\Phi^{(4)}[a,b;c,c';x,y]\++\+\frac{ax(1-b)}{(1-c)}\Phi^{(4)}[aq,bq;cq, c';x,y]\\
                        \++\+\frac{ay(1-b)}{(1-c')}\Phi^{(4)}[aq,bq; c,c'q;xq,y];\\
\Phi^{(4)}[aq^{-1},b;c,c';x,y]=\Phi^{(4)}[a,b,b';c,c';x,y]\+-\+\frac{ax(1-b)}{q(1-c)}\Phi^{(4)}[a,bq;cq,c';x,y]\\
                        \+-\+\frac{ay(1-b)}{q(1-c')}\Phi^{(4)}[a,bq;c,c'q;xq,y],
\enm
we can establish the recursion formulas of $\Phi^{(4)}$ with parameter $a$ in two expressions in the following two theorems.

\begin{thm} [The recursion formulas with parameter $a$]
\bnm
\Phi^{(4)}[aq^n;b;c,c';x,y]=\Phi^{(4)}[a;b;c,c';x,y]\++\+\frac{ax(1-b)}{(1-c)}\sum_{k=1}^nq^{k-1}\Phi^{(4)}[aq^k;b q;c q,c';x,y]\\
\++\+ \frac{ay(1-b)}{(1-c')}\sum_{k=1}^n q^{k-1}\Phi^{(4)}[aq^k;bq;c, c'q;x q,y];\\
\Phi^{(4)}[aq^{-n};b;c,c';x,y]=\Phi^{(4)}[a;b;c,c';x,y]\+-\+\frac{ax(1-b)}{(1-c)}\sum_{k=1}^n q^{-k}\Phi^{(4)}[aq^{1-k};b q;c q,c';x,y]\\
\+-\+ \frac{ay(1-b)}{(1-c')}\sum_{k=1}^n q^{-k}\Phi^{(4)}[aq^{1-k}; bq; c, c'q;x q, y].
\enm
\end{thm}

\begin{thm} [The recursion formulas with parameter $a$ in another expression]
\bnm
\Phi^{(4)}[aq^n,b,b';c,c';x,y]\+=\+\sum_{k=0}^n\sum_{i=0}^{k}\binq{n}{k}\binq{k}{i}
\frac{(b;q)_{k}}{(c;q)_{k-i}(c';q)_{i}}q^{2{k\choose 2}}a^k x^{k-i}y^i\\
\+\times\+\Phi^{(4)}[aq^k;bq^{k};cq^{k-i},c'q^i;xq^i,y];\\
\Phi^{(4)}[aq^{-n};b;c,c';x,y]\+=\+\sum_{k=0}^n\sum_{i=0}^{k}\binq{n}{k}\binq{k}{i}
\frac{(b;q)_{k}}{(c;q)_{k-i}(c';q)_{i}}q^{{k\choose 2}-nk}(-a)^k x^{k-i}y^{i}\\
\+\times\+\Phi^{(4)}[a,bq^{k};cq^{k-i},c'q^{i};xq^{i},y];
\enm
\end{thm}

By the following contiguous relation
\bnm
\+\+\Phi^{(4)}[a,b;cq,c; x,y]=(1-c)\Phi^{(4)}[a,b;c,c';x,y]+c\Phi^{(4)}[a,b;cq,c';xq,yq];\\[2mm]
\+\+\Phi^{(4)}[a,b;cq^{-1};x,y]=\frac{1}{1-q/c}\Phi^{(4)}[a,b;c,c';xq,yq]-\frac{q/c}{1-q/c}\Phi^{(4)}[a,b;c,c';x,y].
\enm
we can get the following results.
\begin{thm} [The recursion formulas with parameter $c$]
\bnm
\+\+\Phi^{(4)}[a,b,b';cq^{-n},c';x,y]
=\frac{1}{(q/c;q)_n}\sum_{k=0}^n\binq{n}{k}(-c)^{k-n}q^{{{n+1-k}\choose 2}-1}\:\Phi^{(4)}[a;b;c,c';xq^k,y];\\
\+\+\Phi^{(4)}[a,b,b';cq^{n},c';x,y]
=\sum_{k=0}^n\binq{n}{k}c^{k}q^{2{{k}\choose 2}}(cq^k;q)_{n-k}\:\Phi^{(4)}[a; b,b';cq^k,c';xq^k,yq^k].
\enm
\end{thm}

The results in this part can be obtained similarly as the recursion formula of $\Phi^{(1)}$ in the first part. Here, we will present with no details.

In fact, by contiguous relations, we can establish the recursion formulas of multiply $q$-hypergeometric functions.
The interested author can do by themselves.

\end{document}